\documentclass[a4paper,12pt,twoside]{article}
\usepackage{latexsym}
\usepackage{a4wide}
\usepackage{amscd}
\usepackage{graphics}
\usepackage{graphicx}
\usepackage[english]{varioref}
\usepackage{amsmath}
\usepackage{amssymb}
\usepackage{mathrsfs}
\usepackage{amsthm}
\usepackage{amssymb,tikz}

\usepackage{bbm}
\usepackage{color}
\usepackage{accents}
\usepackage{enumerate}
\usepackage[utf8x]{inputenc}
\usepackage[T1]{fontenc}
\usepackage[english]{babel}
\usepackage{amsfonts, amscd, amsmath, amssymb, amsthm, mathrsfs, mathtools}
\usepackage{braket}
\input xy
\xyoption{all}
\usepackage[utf8x]{inputenc}
\usepackage[T1]{fontenc}
\usepackage[english]{babel}
\usepackage{amsfonts, amscd, amsmath, amssymb, amsthm, mathrsfs, mathtools}
\usepackage{graphicx}
\usepackage{subfig}
\usepackage{braket}
\usepackage[left=2.5cm,right=2.5cm,top=2.5cm,bottom=2.5cm]{geometry}
\usepackage{fancyhdr}
\theoremstyle{definition}

\newtheorem{remark}{Remark}[section]

\newtheorem{esempio}{Example}[section]

\theoremstyle{plain}
\newtheorem{definizione}{Definition}[section]
\newtheorem{teorema}{Theorem}[section]
\newtheorem{proposizione}{Proposition}[section]
\newtheorem{lemma}{Lemma}[section]
\newtheorem{corollario}{Corollary}[section]

\newcommand{\numberset}{\mathbb}

\newcommand{\R}{\numberset{R}}

\newcommand{\Z}{\numberset{Z}}

\usepackage{bookmark,hyperref}

\title{A Conley-type decomposition \\ of the strong chain recurrent set}

\author{\hspace{-1cm}{OLGA BERNARDI${\,
}^{1}$ \quad ANNA FLORIO${\,
}^{2}$                  }
\vspace{0.3cm}
\\
\hspace{-1cm}${\ }^{1}$ Dipartimento di Matematica ``Tullio Levi-Civita'',
Universit\`a di Padova,\\
\hspace{-1cm}Via Trieste, 63 - 35121 Padova, Italy
\\ 
\\
\hspace{-1cm}${\ }^{2}$ Laboratoire de Mathématiques d'Avignon, Avignon Université, \\
\hspace{-1cm} 84018 Avignon, France \\}

\date{ }

\begin{document}

\maketitle

\begin{abstract}
\noindent For a continuous flow on a compact metric space, the aim of this paper is to prove a Conley-type decomposition of the strong chain recurrent set. We first discuss in details the main properties of strong chain recurrent sets. We then introduce the notion of strongly stable set as an invariant set which is the intersection of the $\omega$-limits of a specific family of nested and definitively invariant neighborhoods of itself. This notion strengthens the one of stable set; moreover, any attractor results strongly stable. We then show that strongly stable sets play the role of attractors in the decomposition of the strong chain recurrent set; indeed, we prove that the strong chain recurrent set coincides with the intersection of all strongly stable sets and their complementary.
\end{abstract}

\section{Introduction}

\indent Let $\phi = \{ \phi_t \}_{t \in \R }$ be a continuous flow on a compact metric space $(X,d)$. In the celebrated paper \cite{conl33}, Charles Conley introduced the notion of \textit{chain recurrence}. A point $x\in X$ is said to be chain recurrent if for any $\varepsilon >0$ and $T > 0$, there exists a finite sequence $(x_i,t_i)_{i=1,...,n}\subset X\times\R$ such that $t_i\geq T$ for all $i$, $x_1 = x$ and setting $x_{n+1}:=x$, we have 
\begin{equation} \label{prima intro}
d(\phi_{t_i}(x_i),x_{i+1}) < \varepsilon \qquad \forall i = 1,\ldots,n.
\end{equation}
The set of chain recurrent points, denoted by $\mathcal{CR}(\phi)$, results closed, invariant and it (strictly) contains the set of non-wandering points. In the same paper, Conley defined attractor-repeller pairs as follows. We know that a set $A \subset X$ is an attractor if it is the $\omega$-limit of a neighborhood $U$ of itself, that is $A = \omega(U)$. Similarly, a set which is the $\alpha$-limit of one of its neighborhoods is a repeller. 
Given an attractor $A$, he proved that the set $A^*$, constituted by the points $x \in X$ whose $\omega$-limit has empty intersection with $A$, is a repeller, hence called the complementary repeller of $A$. Successively, Conley investigated the refined link between chain recurrence and attractors, showing that
\begin{equation} \label{A A star}
{\cal{CR}}(\phi) = \bigcap \left\{ A \cup A^*: \ A \text{ is an attractor}\right\}.
\end{equation}
As an outcome, he made explicit a continuous Lyapunov function which is constant on the connected components of the chain recurrent set and strictly decreasing outside. The construction of such a function is essentially based on the closure and countability of the sets $A \cup A^*$. This result is considered so important that it is sometimes called the fundamental theorem of dynamical systems, see e.g. \cite{norton}. \\
\indent In the same year, Robert Easton in \cite{east41} defined strong chain recurrent points substituting condition (\ref{prima intro}) by
\begin{equation} \label{seconda intro}
\sum_{i=1}^n d(\phi_{t_i}(x_i),x_{i+1}) < \varepsilon.
\end{equation}
Moreover, he connected this stronger notion to the property for the corresponding Lipschitz first integrals to be constants. The set of strong chain recurrent points is denoted by $\mathcal{SCR}(\phi)$. \\
For $Y \subset X$, we define $\Omega(Y,\varepsilon,T)$ to be the set of $x \in X$ such that there is a strong $(\varepsilon,T)$-chain from a point $y \in Y$ to $x$, and  
$$\bar{\Omega} (Y) := \bigcap_{\varepsilon > 0, \ T > 0} \Omega (Y,\varepsilon,T).
$$
\indent Let $f$ be a homeomorphism on a compact metric space $(X,d)$. We remark that Albert Fathi and Pierre Pageaut recently gave a new insight into the subject (see \cite{fathi2015} and \cite{pageault}). Their point of view is very different from the one of Conley's original work and it is inspired by the celebrated work of A. Fathi in Weak KAM Theory (see \cite{fathiWeakKAM}). This approach is based on a re-interpretation from a purely variational perspective of chain recurrent and strong chain recurrent sets. One of their fundamental results is the construction of a Lipschitz continuous Lyapunov function which is constant on the connected components of the strong chain recurrent set and strictly decreasing outside (see Proposition 4.2 in \cite{pageault}). 
An adaptation of Fathi and Pageault's techniques for a flow $\phi = \{\phi_t\}_{t \in \R}$ which is Lipschitz continuous for every $t \ge 0$, uniformly for $t$ on compact subsets of $[0,+\infty)$, can be found in \cite{abbbercar}. \\
\indent This paper is devoted to investigate --in the original spirit of Conley's work-- the structure of the strong chain recurrent set $\mathcal{SCR}(\phi)$ for a continuous flow on a compact metric space. In order to describe our results in some more details, we need to introduce the notion of strongly stable set. The next definition comes after a careful understanding of how $\mathcal{SCR}(\phi)$ dynamically strengthens $\mathcal{CR}(\phi)$. \\
 \\
DEFINITION. \textit{A closed set $B \subset X$ is strongly stable if there exist $\rho > 0$, a family $(U_{\eta})_{\eta \in (0,\rho)}$ of closed nested neighborhoods of $B$ and a function 
$$(0,\rho) \ni \eta \mapsto t(\eta) \in (0,+\infty)$$
bounded on compact subsets of $(0,\rho)$, such that 
\begin{itemize}
\item[$(i)$] For any $0<\eta < \lambda < \rho$, 
$\{x \in X: \ d(x,U_{\eta}) < \lambda -\eta \} \subseteq U_{\lambda}$.
\item[$(ii)$] $B = \bigcap_{\eta \in (0,\rho)} \omega(U_{\eta})$.
\item[$(iii)$] For any $0 < \eta < \rho$, $cl \left( \phi_{[t(\eta),+\infty)}(U_{\eta}) \right) \subseteq U_{\eta}$.
\end{itemize}}
\noindent \noindent In particular, $B$ is (closed and) invariant and so $\omega(B) = B$. The above notion --see Remark \ref{osse Hurley}-- strengthens the one of stable set, or Lyapunov stable set (\textit{cfr.} \cite{akinstable} (Page 1732) and \cite{akinhurleystable} (Paragraph 1.1)). Moreover, as shown in Example \ref{caratte attractors}, any attractor results strongly stable. Analogously to the case of attractors, given a strongly stable set $B$, we define its complementary $B^\bullet$ to be the set of points $x \in X$ whose $\omega$-limit has empty intersection with $B$. We first prove the next result. \\ \\
THEOREM 1. \textit{If $Y \subset X$ is closed, then $\bar{\Omega}(Y)$ is the intersection of all strongly stable sets in $X$ which contain $\omega(Y)$, that is 
\begin{equation*}
\bar{\Omega}(Y) = \bigcap \left\{ B: \ B\text{ is strongly stable and } \omega(Y)\subseteq B \right\}.
\end{equation*}} \\
\noindent The proof is divided into two main parts, corresponding to the two inclusions. The most delicate point is understanding and making explicit the family of strongly stable sets which actually decompose $\bar{\Omega}(Y)$. \\
\noindent As an outcome, we establish that strongly stable sets play the role of attractors in the decomposition of the strong chain recurrent set. More precisely, we prove that $\mathcal{SCR}(\phi)$ admits the following Conley-type decomposition: \\ \\
THEOREM 2. \textit{If $\phi : X \times \R \rightarrow X$ is a continuous flow on a compact metric space, then the strong chain recurrent set is given by
\begin{eqnarray} \label{nostra intro}
{\cal{SCR}}(\phi) = \bigcap \{ B \cup B^{\bullet} :\ B \text{ is strongly stable\,}\}.\label{seconda decomposizione teorema conclusivo}
\end{eqnarray}} \\
\noindent We notice that equality (\ref{nostra intro}) turns into (\ref{A A star}) in the case where $\mathcal{SCR}(\phi) = \mathcal{CR}(\phi)$. \\
\indent The paper is organized as follows. Section \ref{prel} is devoted to preliminaries. In Section \ref{relazioni} we discuss in details the main properties of strong chain recurrent sets. The proofs of the theorems above are given in Section \ref{sezione ris}; we conclude with an example in a $1$-dimensional case. 

~\newline
\textbf{Acknowledgements.} \noindent The authors are very grateful to prof. Marie-Claude Arnaud for the careful reading of the preprint and for the precious advices and suggestions, which also enable us to improve the result of the theorems. The authors are sincerely in debt with the anonymous referee who proposed to use the name \textit{``strongly stable set''} and explained the relation between strongly stable and stable sets. Moreover, this referee gave a fundamental contribution in order to re-organize the first version of the paper. 
\section{Preliminaries} \label{prel}

Throughout the whole paper, $\phi: \R \times X \rightarrow X$ is a continuous flow on a compact metric space $(X,d)$. We use the notation $\phi(t,x) = \phi_t(x)$. This section is devoted to introduce some standard notions. Through the whole paper, $cl$ denotes the closure and $B(x,\varepsilon)$ the open ball with center $x \in X$ and radius $\varepsilon > 0$.
\\
\indent A set $C \subset X$ is called invariant if for any $x \in C$ it holds $\phi_t(x) \in C$ for all $t \in \R$; in such a case, $\phi_t(C) = C$ for any $t \in \R$. A set $C \subset X$ is called forward invariant if for any $x \in C$ one has $\phi_t(x) \in C$ for all $t \ge 0$. We now recall the well-known definition of limit sets, which describe the long term behavior of a subset $Y \subset X$ subjected to the flow $\phi$.
\begin{definizione} {\em ($\omega$-limit and $\alpha$-limit sets)} \par\noindent
The limit sets of $Y \subset X$ are given by
$$
\omega(Y) := \bigcap_{T \ge 0} cl\left( \phi_{[T,+\infty)}(Y) \right) \qquad \text{and} \qquad \alpha(Y) := \bigcap_{T \ge 0} cl\left( \phi_{(-\infty,-T]}(Y) \right).
$$
\end{definizione}
\noindent We point out that the above definition can be formulated in the more general case where $X$ is a Hausdorff space (not necessarily compact), see e.g. \cite{hurleyNoncompact}, \cite{hurley1991} and \cite{sal85}. \\
Clearly, if $Y \subseteq Z$ then $\omega(Y) \subseteq \omega(Z)$. 


In the sequel we collect some useful facts about $\omega(Y)$, in the case where $(X,d)$ is a compact metric space. Equivalent properties hold for $\alpha(Y)$. We refer to \cite{ding2005}, \cite{sal85} and \cite{teschi12} for an exhaustive treatment of the subject. \begin{proposizione} \label{eccola} Let $Y \subset X$.
\begin{itemize}
\item[$(a)$] $\omega(Y)$ is compact and nonempty. If in addition $Y$ is connected, then $\omega(Y)$ is connected. 
\item[$(b)$] $\omega(Y)$ is the maximal invariant set in $cl \left( \phi_{[0,+\infty)}(Y) \right)$. 
\item[$(c)$] $x \in \omega(Y)$ if and only if there are sequences $y_n \in Y$ and $t_n \in \mathbb{R}$, $t_n \to +\infty$, such that $\lim_{n \to +\infty} \phi_{t_n}(y_n) = x$. 
\item[$(d)$] For any neighborhood $U$ of $\omega(Y)$ there exists a time $\bar{t} > 0$ such that  
$$\phi_{[\bar{t},+\infty)}(Y) \subset U.$$
\end{itemize}
\end{proposizione}
\noindent In particular, a closed set $C$ is invariant if and only if $C = \omega(C)$. Moreover, for any $Y \subset X$, $\omega(\omega(Y)) = \omega(Y)$. \\
\indent In the celebrated paper \cite{conl33}, Conley introduced attractors, repellers and a weak form of recurrence, called chain recurrence. 
\begin{definizione} {\em (Attractor and repeller)} \\
A subset $A\subset X$ is called {\it{attractor}} if there exists a neighborhood $U$ of $A$ such that $\omega(U) =A$. Similarly, a set which is the $\alpha$-limit of a neighborhood of itself is called repeller.
\end{definizione}
\noindent Given an attractor $A \subset X$, let us define
\begin{equation} \label{complementary}
A^* := \{ x\in X :\ \omega(x)\cap A=\emptyset \} \qquad \text{and} \qquad C(A,A^*) := X \setminus (A\cup A^*).
\end{equation}
The sets $A^*$ and $C(A,A^*)$ are called respectively the complementary (or dual) repeller of $A$ and the connecting orbit of the pair $(A,A^*)$. The next properties hold (see \cite{akin46}, \cite{conl33} and \cite{sal85}). 
\begin{proposizione}\label{proprieta attrattore}
Let $A\subset X$ be an attractor. 
\begin{itemize}
\item[$(a)$] $A^*$ is a repeller.
\item[$(b)$] A point $x \in C(A,A^*)$ if and only if $\omega(x) \subseteq A$ and $\alpha(x) \subseteq A^*$.
\item[$(c)$] For any $x\in X$, either $\omega(x) \subseteq A$ or $x\in A^*$. 
\end{itemize}
\end{proposizione}
~\newline
\begin{definizione}{\em (Chain recurrence)} \label{scrdef1}  \par\noindent
\begin{itemize}
\item[$(i)$] Given $x,y\in X$, $\epsilon > 0$ and $T > 0$, an $(\epsilon,T)$-chain from $x$ to $y$ is a finite sequence $(x_i,t_i)_{i=1,\dots, n} \subset X \times \mathbb{R}$ such that $t_i \geq T$ for all $i$,  $x_1 = x$ and setting $x_{n+1} := y$, we have
\begin{equation} \label{CR}
d(\phi_{t_i}(x_i),x_{i+1}) < \epsilon
\end{equation}
$\forall i = 1,\ldots,n$.
\item[$(ii)$] A point $x \in X$ is said to be chain recurrent if for all $\epsilon > 0$ and $T > 0$, there exists an $(\epsilon,T)$-chain from $x$ to $x$. The set of chain recurrent points is denoted by $\mathcal{CR}(\phi)$.  
\end{itemize}
\end{definizione}
\noindent Expressed in words, a point is chain recurrent if it returns to itself by following the flow for a time great as you like, and by allowing an arbitrary number of small jumps.
\noindent In the same paper \cite{conl33}, Conley described the structure of the chain recurrent set $\mathcal{CR}(\phi)$ in terms of attractor-repeller pairs and, as an outcome, proved the intimate relation between chain recurrence and Lyapunov functions. \\
\indent 
We first recall the notion of continuous Lyapunov function and then we state the so-called Conley's decomposition theorem. 
\begin{definizione} {\em (Lyapunov function)} \\
A function $h: X \to \mathbb{R}$ is a Lyapunov function for $\phi$ if $h \circ \phi_t \le h$ for every $t \ge 0$.  
\end{definizione}
\begin{teorema} {\em (Conley's decomposition theorem)} \label{CON} \\
Let $\phi: \R \times X \rightarrow X$ be a continuous flow on a compact metric space.
\begin{itemize}
\item[$(i)$] The chain recurrent set is given by
\begin{equation} \label{altra}
{\cal{CR}}(\phi) = \bigcap \left\{ A \cup A^*: \ A \text{ is an attractor\,}\right\}.
\end{equation}
\item[$(ii)$] There exists a continuous Lyapunov function $h: X \to \mathbb{R}$ which is constant on the connected components of the chain recurrent set and strictly decreasing outside, that is
\begin{eqnarray*}
h \circ \phi_t (x) < h(x)  & \qquad  & \text{ for all } x \in X \setminus {\cal{CR}}(\phi) \text{ and } t > 0 \\
h \circ \phi_t (x) = h(x)  & \qquad  & \text{ for all } x \in {\cal{CR}}(\phi) \text{ and } t \ge 0.
\end{eqnarray*}
\end{itemize}
\end{teorema}
\indent In the same year, Robert Easton in \cite{east41} strengthened chain recurrence by substituting (\ref{CR}) with the condition that the sum of the jumps is arbitrarily small, as explained in the next
\begin{definizione}{\em (Strong chain recurrence)} \label{scrdef}  \par\noindent
$(i)$ Given $x,y\in X$, $\epsilon > 0$ and $T > 0$, a strong $(\epsilon,T)$-chain from $x$ to $y$ is a finite sequence $(x_i,t_i)_{i=1,\dots, n} \subset X \times \mathbb{R}$ such that $t_i \geq T$ for all $i$,  $x_1 = x$ and setting $x_{n+1} := y$, we have
\begin{equation} \label{SCRR}
\sum_{i=1}^n d(\phi_{t_i}(x_i),x_{i+1}) < \epsilon.
\end{equation}
$(ii)$ A point $x \in X$ is said to be strong chain recurrent if for all $\epsilon > 0$ and $T > 0$, there exists a strong $(\epsilon,T)$-chain from $x$ to $x$. The set of strong chain recurrent points is denoted by $\mathcal{SCR}(\phi)$.  \par\noindent 
\end{definizione}
\noindent Clearly, the set of fixed points, and --more generally speaking-- the one of periodic points, is contained in $\mathcal{SCR}(\phi) \subseteq \mathcal{CR}(\phi)$. Moreover, $\mathcal{SCR}(\phi)$ and $\mathcal{CR}(\phi)$ are easily seen to be invariant and closed (see e.g. \cite{conl33}, \cite{akin46} and \cite{zz00}).
\section{Strong chain recurrence in compact metric spaces} \label{relazioni}

For fixed $\varepsilon > 0$ and $T > 0$, let us define 
\begin{equation} \label{H1}
\Omega(Y,\varepsilon,T) = \{x \in X: \ \text{there is a strong $(\varepsilon,T)$-chain from a point of $Y$ to $x$} \}
\end{equation} and 
\begin{equation} \label{hurley 1}
\bar{\Omega}(Y,\varepsilon,T) = \bigcap_{\eta > 0} \Omega(Y,\varepsilon + \eta,T).
\end{equation}
In the next lemmas, we collect the main properties of the sets $\Omega(Y,\varepsilon,T)$ and $\bar{\Omega}(Y,\varepsilon,T)$. 
\begin{lemma} $\Omega(Y,\varepsilon,T)$ is open.
\end{lemma}
\begin{proof} Let $\varepsilon_1 \in (0,\varepsilon)$. If there exists a strong $(\varepsilon_1,T)$-chain from a point $y \in Y$ to $x \in X$, then replacing $x$ by any point $x_1 \in B(x,\varepsilon - \varepsilon_1)$ we obtain a strong $(\varepsilon,T)$-chain from $y$ to $x_1$. 
\end{proof} 
\begin{lemma}\label{lemma omegabar chiuso} $\bar{\Omega}(Y,\varepsilon,T)$ is closed.
\end{lemma}
\begin{proof} 
Let $x\in cl(\bar{\Omega}(Y,\varepsilon,T))$ and $\eta>0$. Then there exist $x_1\in\bar{\Omega}(Y,\varepsilon,T)$ such that $d(x,x_1)<\frac{\eta}{2}$ and a strong $(\varepsilon+\frac{\eta}{2},T)$-chain from a point $y\in Y$ to $x_1$. Replace $x_1$ by $x$ and obtain a strong $(\varepsilon+\eta,T)$-chain from $y$ to $x$.
\end{proof}
\begin{lemma} \label{lemma 2} $\bar{\Omega}(Y,\varepsilon,T) = \bar{\Omega}(cl(Y),\varepsilon,T)$.
\end{lemma}
\begin{proof} 
Let $\eta>0$. If $x\in\bar{\Omega}(cl(Y),\varepsilon,T)$ then there exists a strong $(\varepsilon+\frac{\eta}{2},T)$-chain $(x_i,t_i)_{i=1,\dots,n}$ from a point $x_1=y\in cl(Y)$ to $x_{n+1}=x$. Let $\delta>0$ be a $\frac{\eta}{2}$ modulus of uniform continuity of $\phi_{t_1}$ and let $y_1\in Y$ such that $d(y,y_1)<\delta$. Replacing $y$ by $y_1$, we obtain a strong $(\varepsilon+\eta,T)$-chain from $y_1$ to $x$.
 \end{proof}
\begin{lemma} \label{lemma 3} $\{x \in X: \ d(x,\bar{\Omega}(Y,\varepsilon,T)) < \eta \} \subseteq \Omega(Y,\varepsilon + \eta,T)$
\end{lemma}
\begin{proof} 
Let $\bar{x}\in\bar{\Omega}(Y,\varepsilon,T)$ and $d(x,\bar{x})<\eta$. Define $\zeta=\eta-d(x,\bar{x})>0$. There exists a strong $(\varepsilon+\zeta,T)$-chain from a point $y\in Y$ to $\bar{x}$. Replace $\bar{x}$ by $x$ and obtain a strong $(\varepsilon+\eta,T)$-chain from $y$ to $x$.
\end{proof}
\begin{lemma}\label{lemma condizione iii ss}
$cl\left\{\phi_{[T,+\infty)}(\bar{\Omega}(Y,\varepsilon,T))\right\} \subseteq \bar{\Omega}(Y,\varepsilon,T)$.
\end{lemma}
\begin{proof}
Observe first that, by definition of strong $(\varepsilon,T)$-chain, $\phi_{[T,+\infty)}(\Omega(Y,\varepsilon,T))\subseteq\Omega(Y,\varepsilon,T)$. 
\noindent As a consequence, $cl\left\{\phi_{[T,+\infty)}(\bar{\Omega}(Y,\varepsilon,T)\right\}\subseteq\bar{\Omega}(Y,\varepsilon,T)$. %
\end{proof}
\begin{corollario}\label{lemma 1.1}
$\omega(\bar{\Omega}(Y,\varepsilon,T)) \subseteq \bar{\Omega}(Y,\varepsilon,T)$.
\end{corollario}
\begin{lemma} \label{lemma 4} $\omega(Y) \subseteq \bar{\Omega}(Y,\varepsilon,T)$.
\end{lemma}
\begin{proof} 
Notice that $\phi_{[T,+\infty)}(Y)\subseteq\Omega(Y,\varepsilon,T)$. Indeed, if $y\in Y$ then $(y,t_1)$ (with $t_1\geq T$) is a strong $(0,T)$-chain from $y$ to $\phi_{t_1}(y)$ . Then, $cl\left\{\phi_{[T,+\infty)}(Y)\right\}\subseteq\bar{\Omega}(Y,\varepsilon,T)$ and $\omega(Y)\subseteq\bar{\Omega}(Y,\varepsilon,T)$.
\end{proof}
\begin{lemma} \label{tran} $\Omega(\Omega(Y,\varepsilon_1,T),\varepsilon_2,T) \subseteq \Omega(Y,\varepsilon_1 + \varepsilon_2,T)$.
\end{lemma}
\begin{proof} 
Let $z\in\Omega(\Omega(Y,\varepsilon_1,T),\varepsilon_2,T)$ and let $(x_i,t_i)_{i=1,\dots,n}$ be a strong $(\varepsilon_2,T)$-chain from $x\in\Omega(Y,\varepsilon_1,T)$ to $z$. There exists a strong $(\varepsilon_1,T)$-chain from a point $y\in Y$ to $x$. Concatenating the two chains, we obtain a strong $(\varepsilon_1+\varepsilon_2,T)$-chain from $y$ to $z$.
\end{proof}
\noindent Let us define now
\begin{equation} \label{hurley}
\bar{\Omega}(Y) = \bigcap_{\varepsilon > 0, \ T > 0} \Omega(Y,\varepsilon,T) = \bigcap_{\varepsilon > 0, \ T > 0} \bar{\Omega}(Y,\varepsilon,T).
\end{equation}
By Lemmas \ref{lemma 2} and \ref{lemma 4}, we immediately conclude that
\begin{equation}\label{equazione per lemma}
\omega(Y) \subseteq \bar{\Omega}(Y) = \bar{\Omega}(cl(Y)).
\end{equation}
\begin{esempio} Consider a Cantor set $K$ in $[0,1]$ with $0,1 \in K$. Let $f:  [0,1]\to [0,+\infty)$ be a non negative smooth function whose set of zeroes coincides with the Cantor set. Let $\phi:\R \times [0,1] \rightarrow [0,1]$ be the flow of the vector field
\[
V(x) := f(x) \frac{\partial}{\partial x}.
\]
We denote by $\lambda$ the Lebesgue measure on the interval. If $\lambda(K) = 0$, then $\bar{\Omega}(x) =[x,1]$ for every $x \in K$. If $\lambda(K) = \delta > 0$, then $\bar{\Omega}(x) = \{x\}$ for every $x \in K$. 
\end{esempio}
\noindent Let us define
$$\mathcal{SP}(X) = \{(y,x) \in X \times X: \ x \in \bar{\Omega}(y)\}.$$
\noindent The properties of the relation $\mathcal{SP}(X)$ are illustrated in the next proposition, see also \cite{wiseman} (Definition 3.2) and \cite{zz00} (Definition 2.3 and Theorems 2.1, 2.2, 2.3). We refer to \cite{conl33} (Page 36, Section 6) for the analogous result in the chain recurrent case. 

\begin{proposizione}\label{sprelation}
${\cal{SP}}(X)$ is a transitive, closed relation on $X$. Moreover, if $(y,x)\in{\cal{SP}}(X)$ and $t,s\in \R$, then $(\phi_t(y),\phi_s(x))\in{\cal{SP}}(X)$.
\end{proposizione}
\begin{proof}
The transitivity of the relation ${\cal{SP}}(X)$ follows directly from Lemma \ref{tran}. \\
\indent In order to prove that the relation ${\cal{SP}}(X)$ is closed, let $(\bar{y},\bar{x})$ be a limit point of ${\cal{SP}}(X)$. By the continuity of $\phi_T$ at $\bar{y}$, for any $\varepsilon > 0$ there exists $\varepsilon_1 \in \left(0,\frac{\varepsilon}{4}\right)$ such that 
$$
d(\bar{y},z)<\varepsilon_1\quad\Rightarrow\quad d(\phi_T(\bar{y}),\phi_T(z))<\frac{\varepsilon}{4}.
$$
Consequently, since $(\bar{y},\bar{x})$ is a limit point of ${\cal{SP}}(X)$, there exists $(y,x)\in{\cal{SP}}(X)$ such that $d(\bar{y},y)<\varepsilon_1$ and $d(\bar{x},x)<\varepsilon_1$ implying 
\begin{equation}\label{eq continuous flow}
d(\phi_T(\bar{y}),\phi_T(y))<\frac{\varepsilon}{4}.
\end{equation}
\noindent For any $\varepsilon > 0$ and $T > 0$ we construct a strong $(\varepsilon,T)$-chain from $\bar{y}$ to $\bar{x}$ by concatenation of three strong chains. \\
First, by inequality \eqref{eq continuous flow}, we build a strong $(\frac{\varepsilon}{4},T)$-chain from $\bar{y}$ to $\phi_T(y)$ made up of a single jump:
$$
\begin{cases}
x_1 := \bar{y} \qquad t_1 := T \\
x_2 := \phi_T(y).
\end{cases}
$$
Moreover, since $(y,x)\in{\cal{SP}}(X)$, there exists a strong $(\varepsilon_1,2T)$-chain $(x_i,t_i)_{i=1,\dots,n}$ from $y$ to $x$ (so with $x_{n+1}=x$). We consider the strong $(\varepsilon_1,T)$-chain $(z_i,s_i)_{i=1,\dots,n-1}$ from $\phi_T(y)$ to $x_n$ (hence with $z_n = x_n$):
\[
\begin{cases}
z_1 := \phi_T(y) & s_1 := t_1-T\\
z_i := x_i & s_i := t_i \qquad\qquad \forall i = 2,\dots,n-1\\
z_{n} := x_n
\end{cases}
\]
where $\varepsilon_1 \in \left(0,\frac{\varepsilon}{4}\right)$. \\
Finally, we exhibit the following strong $(\frac{\varepsilon}{2},T)$-chain from $x_n$ to $\bar{x}$:
$$
\begin{cases}
\tilde{x}_1 := x_n \qquad \tilde{t}_1 := t_n \\
\tilde{x}_2 := \bar{x}.
\end{cases}
$$
Indeed, $d(\phi_{t_n}(x_n),\bar{x}) \le d(\phi_{t_n}(x_n),x) + d(x,\bar{x}) < 2\varepsilon_1 < \frac{\varepsilon}{2}$. \\
\noindent By gluing the three strong chains above, we obtain a strong $(\varepsilon,T)$-chain from $\bar{y}$ to $\bar{x}$ proving that $(\bar{y},\bar{x})\in{\cal{SP}}(X)$. \\
\indent Finally, we prove that if $(y,x)\in{\cal{SP}}(X)$ and $t,s\in\R$, then $(\phi_t(y),\phi_s(x))\in{\cal{SP}}(X)$. This means that for any $\varepsilon > 0$ and $T > 0$ there exists a strong $(\varepsilon,T)$-chain from $\phi_t(y)$ to $\phi_s(x)$. \noindent We point out that, since $(y,x)\in{\cal{SP}}(X)$, for any $\varepsilon_1>0$ and $T > 0$, there exists a strong $(\varepsilon_1,T+|t|+|s|)$-chain $(x_i,t_i)_{i=1,\dots,n}$ from $y$ to $x$.
Moreover, since $\phi_s$ is continuous at $x$, for any $\varepsilon > 0$ there exists $\varepsilon_1 \in \left(0,\frac{\varepsilon}{2}\right)$ such that
\[
d(z,x)<\varepsilon_1\quad\Rightarrow\quad d(\phi_{s}(z),\phi_s(x))<\frac{\varepsilon}{2}.
\]
\noindent Let $\varepsilon_1 \in \left(0,\frac{\varepsilon }{2} \right)$. Hence, for any $\varepsilon > 0$ and $T > 0$, we exhibit the following strong $(\varepsilon,T)$-chain $(z_i,s_i)_{i=1,\dots,n}$ from $\phi_t(y)$ to $\phi_s(x)$:
\[
\begin{cases}
z_1 := \phi_t(y) & s_1 := t_1-t\\
z_i := x_i & s_i := t_i \qquad\qquad \forall i=2,\dots,n-1\\
z_n := x_n & s_n := t_n+s\\
z_{n+1} := \phi_s(x).
\end{cases}
\]
\noindent In fact:
\[
\sum_{i=1}^n d(\phi_{s_i}(z_i),z_{i+1}) = d(\phi_{t_1-t}\circ\phi_t(y),x_2) + \sum_{i=2}^{n-1} d(\phi_{t_i}(x_i),x_{i+1}) + d(\phi_{t_n+s}(x_n),\phi_s(x))=
\]
\[
= \sum_{i=1}^{n-1} d(\phi_{t_i}(x_i),x_{i+1}) + d(\phi_{t_n+s}(x_n),\phi_s(x)) < \varepsilon_1 + \frac{\varepsilon}{2} < \varepsilon
\]
since $d(\phi_{t_n}(x_n),x)<\varepsilon_1$. 
\end{proof}
\begin{corollario}\label{relazione chiusura orbita} If $(y,x) \in \mathcal{SP}(X)$, $y_1 \in cl \{ \phi_{\R}(y)\}$, $x_1 \in cl \{ \phi_{\R}(x)\}$, then $(y_1,x_1) \in \mathcal{SP}(X)$.
\end{corollario}
For $Y \subseteq X$, let now
$$
\mathcal{SP}_Y(X) = \{x\in X : \ \exists y \in Y \text{ such that } x\in\bar{\Omega}(y)\}.
$$
\begin{lemma}\label{lemma sp}
If $Y \subset X$ is closed, then $\mathcal{SP}_Y(X) = \bar{\Omega}(Y)$.  
\end{lemma}
\begin{proof} 
If $x\in\bar{\Omega}(Y)$ then there exists a sequence of strong $(\varepsilon_k,T_k$)-chains with $\varepsilon_k\rightarrow 0$, $T_k\rightarrow+\infty$ with initial point $y_k\in Y$ and final point $x$. Up to subsequences, assume that $y_k$ converges to $y\in Y$. Fix $T > 0$ and $\varepsilon>0$ and choose $k$ so that $\varepsilon_k<\frac{\varepsilon}{2}$, $T_k>2T$ and $d(\phi_T(y_k),\phi_T(y))<\frac{\varepsilon}{2}$. Let $(x_i,t_i)_{i=1,\dots,n}$ be a strong $(\varepsilon_k,T_k)$-chain from $x_1=y_k$ to $x_{n+1}=x$. Then $(z_i,s_i)_{i=1,\dots,n+1}$ with
$$
\begin{cases}
z_1=y & s_1=T \\
z_2=\phi_T(y_k) & s_2=t_1-T\\
z_i=x_{i-1} & s_i=t_{i-1}\qquad\qquad \forall i=3,\dots,n+1\\
z_{n+2}=x
\end{cases}
$$
is a strong $(\varepsilon,T)$-chain from $y$ to $x$.
\end{proof}
\begin{lemma}\label{lemma 1.2}
If $Y \subset X$ is closed, then $\bar{\Omega}(Y)$ is closed and invariant. Consequently, $\omega(\bar{\Omega}(Y))=\bar{\Omega}(Y)$.
\end{lemma}
\begin{proof} 
By definition \eqref{hurley} and Lemma \ref{lemma omegabar chiuso}, $\bar{\Omega}(Y)$ is the intersection of closed sets and therefore it is closed. Invariance of $\mathcal{SP}(X)$ --proved in Proposition \ref{sprelation}-- and Lemma \ref{lemma sp} imply the invariance of $\bar{\Omega}(Y)$. 
\end{proof}
\begin{lemma} \label{lemma 12} If $Y \subset X$ is closed, then $\bar{\Omega}(Y) = \bigcap_{\varepsilon > 0, \ T > 0} \omega(\bar{\Omega}(Y,\varepsilon,T))$.  
\end{lemma}
\begin{proof} 
Since $\bar{\Omega}(Y)\subseteq\bar{\Omega}(Y,\varepsilon,T)$, by Corollary \ref{lemma 1.1} $\omega(\bar{\Omega}(Y))\subseteq\omega(\bar{\Omega}(Y,\varepsilon,T))\subseteq\bar{\Omega}(Y,\varepsilon,T)$. By intersecting, we obtain
$$
\omega(\bar{\Omega}(Y))\subseteq\bigcap_{\varepsilon>0, \ T>0}\omega(\bar{\Omega}(Y,\varepsilon,T))\subseteq\bigcap_{\varepsilon>0, \ T>0}\bar{\Omega}(Y,\varepsilon,T)=\bar{\Omega}(Y).
$$
By Lemma \ref{lemma 1.2}, $\omega(\bar{\Omega}(Y))=\bar{\Omega}(Y)$ and the equality holds.
\end{proof}
\begin{lemma} If $Y \subset X$ is closed, then $\bar{\Omega}(\omega(Y)) = \bar{\Omega}(\bar{\Omega}(Y)) = \bar{\Omega}(Y)$.  
\end{lemma}
\begin{proof}
From one hand, by \eqref{equazione per lemma}, Lemma \ref{lemma sp} and transitivity proved in Proposition \ref{sprelation}, it follows that
\begin{equation*}\label{eq ultimo lemma}
\bar{\Omega}(\omega(Y))\subseteq\bar{\Omega}(\bar{\Omega}(Y))\subseteq\bar{\Omega}(Y).
\end{equation*}
On the other hand, let $x\in \bar{\Omega}(Y) = \mathcal{SP}_Y(X)$ by Lemma \ref{lemma sp} again. Let $y\in Y$ such that $x\in\bar{\Omega}(y)$ and let $y_1\in\omega(y)\subseteq\omega(Y)$. Then --by Corollary \ref{relazione chiusura orbita} -- $(y_1,x)\in\mathcal{SP}(X)$. In other words, $x\in\mathcal{SP}_{\omega(y)}(X) \subseteq \mathcal{SP}_{\omega(Y)}(X)$. Hence, from Lemma \ref{lemma sp}, $x\in\bar{\Omega}(\omega(Y))$. This proves that $\bar{\Omega}(Y)\subseteq\bar{\Omega}(\omega(Y))$.
\end{proof}

\section{Strongly stable sets and strong chain recurrence sets} \label{sezione ris}

We start this section by recalling the notion of (Lyapunov) stable set and then we define strongly stable sets. By Remark \ref{osse Hurley}, every strongly stable set is stable. Moreover --see Example \ref{caratte attractors}-- an attractor results strongly stable. Afterwards, in Theorem \ref{MAIN THEO 1} we prove the decomposition of the set $\bar{\Omega}(Y)$ in terms of strongly stable sets. As an outcome, we finally show that these sets play the role of attractors in the Conley-type decomposition of ${\mathcal{SCR}}(\phi)$. Indeed, Theorem \ref{teorema 5.3.1} is the analogous of point $(i)$ of Theorem \ref{CON}: the strong chain recurrent set of a continuous flow on a compact metric space coincides with the intersection of all strongly stable sets and their complementary.
\begin{definizione} \label{non rep} {\em{(Stable set)}} \\
A closed set $B \subset X$ is stable if it has a neighborhood base of forward invariant sets. 
\end{definizione}
\noindent We notice that the neighborhood base of forward invariant sets can assumed to be closed, since the closure of a forward invariant set is forward invariant.
\begin{lemma} \label{deco} If $B \subset X$ is stable and $\omega(x) \cap B \ne \emptyset$ then $\omega(x) \subseteq B$.
\end{lemma}
\begin{proof} Let $U$ be a closed forward invariant neighborhood of $B$. Since $\omega(x) \cap B \ne \emptyset$, there exists a time $\bar{t} > 0$  such that $\phi_{\bar{t}}(x) \in U$. Since $U$ is closed and forward invariant, we have that $\omega(x) \subset cl \left\{ \phi_{[\bar{t},+\infty)}(x) \right\} \subset U$. As $U$ is arbitrary in the neighborhood base of $B$, we conclude that $\omega(x) \subseteq B$.
\end{proof}
\begin{definizione} \label{non rep} {\em{(Strongly stable set)}} \\ 
A closed set $B \subset X$ is strongly stable if there exist $\rho > 0$, a family $(U_{\eta})_{\eta \in (0,\rho)}$ of closed nested neighborhoods of $B$ and a function 
$$(0,\rho) \ni \eta \mapsto t(\eta) \in (0,+\infty)$$
bounded on compact subsets of $(0,\rho)$, such that 
\begin{itemize}
\item[$(i)$] For any $0<\eta < \lambda < \rho$, 
$\{x \in X: \ d(x,U_{\eta}) < \lambda -\eta \} \subseteq U_{\lambda}$.
\item[$(ii)$] $B = \bigcap_{\eta \in (0,\rho)} \omega(U_{\eta})$.
\item[$(iii)$] For any $0 < \eta < \rho$, $cl \left\{ \phi_{[t(\eta),+\infty)}(U_{\eta}) \right\} \subseteq U_{\eta}$.
\end{itemize}
\end{definizione} 
\begin{remark} \label{osse Hurley} Every strongly stable set is stable. Indeed, if $U$ is an open neighborhood of $B$, then there exists $\eta \in (0,\rho)$ such that $\omega(U_{\eta}) \subset U$. Hence, there exists a time $\bar{t} > 0$ such that $V_{\eta} := cl \left \{ \phi_{[\bar{t},+\infty)} (U_{\eta}) \right\} \subset U$. Clearly, $V_{\eta}$ results closed and forward invariant. 
Since $U_{\eta}$ is a neighborhood of $B$ and $\phi_t$ is continuous, then every $\phi_t(U_{\eta})$ is a neighborhood of $\phi_t(B)=B$. Hence $V_{\eta}$ is a (closed and) forward invariant neighborhood of $B$.
\end{remark}
\begin{esempio} \label{caratte attractors} Every attractor is strongly stable. 
\noindent Indeed, for $\eta > 0$, let $$A_{\eta} := \{x \in X : \ d(x,A) \le \eta \}.$$ 
Condition $(i)$ of Definition \ref{non rep} is clearly satisfied. Moreover, since $A \subset X$ is an attractor, there exists a neighborhood $U$ of $A$ such that $\omega(U) = A$. As a consequence, there exists $\rho > 0$ such that 
$A \subset A_{\rho} \subseteq U$. Therefore, for any $\eta \in (0,\rho)$, $\omega(A_{\eta}) = A$, satisfying then condition $(ii)$ of Definition \ref{non rep}. In addition, there exists a time $t(\eta) > 0$ such that 
\begin{equation} \label{inclu}
cl\left\{ \phi_{[t(\eta),+\infty)} (A_{\eta}) \right\}  \subset int(A_{\eta}).
\end{equation}
We finally prove that the function $(0,\rho) \ni \eta \mapsto t(\eta) \in (0,+\infty)$ is bounded on compact subsets of $(0,\rho)$. Since $\omega(U) = A \subset int(A_{\eta})$ for any $\eta \in (0,\rho)$, we choose the first time $\bar{t}(\eta) > 0$ such that
$$cl \left\{ \phi_{[\bar{t}(\eta),+\infty)} (U) \right\} \subset int(A_{\eta}).$$
Clearly, the function $(0,\rho) \ni \eta \mapsto \bar{t}(\eta) \in (0,+\infty)$ is decreasing and therefore it is bounded on compact subsets of $(0,\rho)$. Since
$$cl \left\{ \phi_{[\bar{t}(\eta),+\infty)}(A_{\eta}) \right\} \subseteq cl \left\{ \phi_{[\bar{t}(\eta),+\infty)}(U) \right\} \subset int(A_{\eta}),$$
we have that $t(\eta) \le \bar{t}(\eta)$, $\forall \eta \in (0,\rho).$ As a consequence, also the function $(0,\rho) \ni \eta \mapsto t(\eta) \in (0,+\infty)$ is bounded on compact subsets of $(0,\rho)$. \\
\noindent We observe that, in the case of attractors, condition $(iii)$ of Definition \ref{non rep} is realized in a stronger way. Indeed, for any $0<\eta<\rho$, one has the stricter inclusion (\ref{inclu}).
\end{esempio} 
\begin{esempio} \label{11}
On the interval $[a,b]$, let us consider the dynamical system of Figure \ref{Figura1}, where the bold line and the arrows denote respectively fixed points and the direction of the flow $\phi$. In such a case, ${\mathcal{SCR}}(\phi) = \{ a, b \}$, the point $b$ is an attractor (and therefore strongly stable). However, every interval $[\alpha,\beta]$ of fixed points with $\beta$ different from $c$ is a strongly stable set but not an attractor.
\begin{figure}[h]
\centering
\begin{tikzpicture}[scale=0.8]
\draw plot [mark=*] coordinates
	{(0,0) (0,5) (0,3.5)};
\draw [ultra thick] (0,2) -- (0,3.5);
\draw (0,0) -- (0,5);
\draw [->] (0,0) -- (0,1);
\draw [->] (0,3.5) -- (0,4.25);
\node at (0.5,0) {a};
\node at (0.5, 5) {b};
\node at (0.5,3.5) {c};
\end{tikzpicture}
\caption{
}
\label{Figura1}
\end{figure}
\end{esempio}
\begin{esempio} \label{esempio inclu} If $Y \subset X$ is closed, then every $\omega(\bar{\Omega}(Y,\varepsilon,T))$ is strongly stable, with $U_{\eta} = \bar{\Omega}(Y,\varepsilon + \eta, T)$. Indeed, Lemma \ref{lemma 3} implies condition $(i)$ of Definition \ref{non rep}. Moreover, from Lemma \ref{lemma condizione iii ss} $cl \left\{\phi_{[T,+\infty)}(\bar{\Omega}(Y,\varepsilon+\eta,T)) \right\} \subseteq \bar{\Omega}(Y,\varepsilon+\eta,T)$ for any $\eta>0$, which is condition $(iii)$ of Definition \ref{non rep}. Finally, from definition \eqref{hurley 1} and the property that $\omega(\bar{\Omega}(Y,\varepsilon+\eta,T)) \subseteq \bar{\Omega}(Y,\varepsilon+\eta,T)$ (see Corollary \ref{lemma 1.1}), we deduce
$$
\omega(\bar{\Omega}(Y,\varepsilon,T))\subseteq\bigcap_{\eta>0}\omega(\bar{\Omega}(Y,\varepsilon+\eta,T))\subseteq\bigcap_{\eta>0}\bar{\Omega}(Y,\varepsilon+\eta,T)=\bar{\Omega}(Y,\varepsilon,T).
$$
We then conclude that $\omega(\bar{\Omega}(Y,\varepsilon,T))=\bigcap_{\eta>0}\omega(\bar{\Omega}(Y,\varepsilon+\eta,T))$, proving condition $(ii)$ of Definition \ref{non rep}.

\end{esempio} 
\indent The next theorem is fundamental in order to prove the Conley-type decomposition of the strong chain recurrent set. The corresponding result in the chain recurrent case is statement C of Section 6 in \cite{conl33}.
\begin{teorema} \label{MAIN THEO 1} If $Y \subset X$ is closed, then $\bar{\Omega}(Y)$ is the intersection of all strongly stable sets in $X$ which contain $\omega(Y)$, that is 
\begin{equation}\label{scrittuta di Omegabar Y}
\bar{\Omega}(Y) = \bigcap \left\{ B: \ B\text{ is strongly stable and } \omega(Y)\subseteq B \right\}.
\end{equation}
\end{teorema}
\begin{proof} $(\subseteq)$ Let $B \subset X$ be strongly stable and $\eta \in (0,\rho)$ be fixed. By hypothesis, $\omega (Y) \subseteq B \subset U_{\eta/2}$. As a consequence, there exists a time $\bar{t}(\eta/2) > 0$ such that
\begin{equation} \label{PRIMA INCLU}
cl \left\{ \phi_{[\bar{t}(\eta/2),+\infty)} (Y) \right\} 
\subseteq U_{\eta/2}.
\end{equation}
Let us define
$$T(\eta) := \max \{ \bar{t}(\eta/2), \max_{\lambda \in [0,\eta/2]} t(\eta/2 + \lambda) \}.$$
We notice that $\max_{\lambda \in [0,\eta/2]} t(\eta/2 + \lambda)$ is achieved since the function $(0,\rho) \ni \eta \mapsto t(\eta) \in (0,+\infty)$ is bounded on compact subsets of $(0,\rho)$. \\
We proceed by proving that every strong $\left(\eta/2,T(\eta)\right)$-chain starting at $Y$ ends in $U_{\eta}$, i.e. $\bar{\Omega}(Y,\eta/2,T(\eta))\subseteq U_{\eta}$. This thesis immediately follows from these facts. By (\ref{PRIMA INCLU}), for every $x_1 \in Y$, if $t_1 \ge T(\eta) \ge \bar{t}(\eta/2)$ then $cl \left( \phi_{[t_1,+\infty)}(x_1) \right) \subseteq U_{\eta/2}$. By condition $(i)$ of Definition \ref{non rep}, if the amplitude of the first jump is $\eta_1 \in (0,\eta/2]$, the point $x_2$ of the chain is contained in $U_{\eta/2 + \eta_1}$. Moreover, for $t_2 \ge T(\eta) \ge \max_{\lambda \in [0,\eta/2]} t(\eta/2 + \lambda)$, by condition $(iii)$ of Definition \ref{non rep}, we have that $cl \left( \phi_{[t_2,+\infty)}(x_2) \right) \subseteq U_{\eta/2 + \eta_1}$. Iterating the above argument and using the fact that the sum of the amplitudes of the jumps is smaller than $\eta/2$, we obtain that
$$\bar{\Omega}(Y,\eta/2,T(\eta)) \subseteq U_{\eta}.$$
Since the above argument holds for any $\eta \in (0,\rho)$, we conclude that
$$\bar{\Omega}(Y) \subseteq \bigcap_{\eta \in (0,\rho)} \omega(U_{\eta}) = B$$
where, in the last equality, we have used condition $(ii)$ of Definition \ref{non rep}. \\
\noindent $(\supseteq)$ We know from Example \ref{esempio inclu} that every $\omega(\bar{\Omega}(Y,\varepsilon,T))$ is strongly stable. Moreover, by Lemma \ref{lemma 12}, $\bar{\Omega}(Y) = \bigcap_{\varepsilon > 0, \ T > 0} \omega(\bar{\Omega}(Y,\varepsilon,T))$. \end{proof}
\indent Given a strongly stable set $B \subset X$, we define the complementary of $B$ to be the set
\begin{equation} \label{B bullet}
B^{\bullet} := \{x \in X : \ \omega(x) \cap B = \emptyset \}.
\end{equation}
The set $B^{\bullet} \subset X$ is invariant and disjoint from $B$ but it is not necessarily closed even if $B$ is closed (see also Paragraph 1.5 in \cite{conARTICOLO}). If $B$ is an attractor then $B^{\bullet}$ coincides with the so-called complementary repeller $B^*$ (see the first formula in (\ref{complementary})). Moreover, every $B \subset X$ strongly stable is stable. Therefore --by Lemma \ref{deco}-- for every $x \in X$, either $\omega(x) \subseteq B$ or $x \in B^{\bullet}$. \\
\indent We finally prove the Conley-type decomposition of the strong chain recurrent set in terms of strongly stable sets. 
\begin{teorema}\label{teorema 5.3.1} 
If $\phi : X \times \R \rightarrow X$ is a continuous flow on a compact metric space, then the strong chain recurrent set is given by
\begin{eqnarray} \label{nostra}
{\cal{SCR}}(\phi) = \bigcap \{ B \cup B^{\bullet} :\ B \text{ is strongly stable\,}\}.\label{seconda decomposizione teorema conclusivo}
\end{eqnarray}
\end{teorema}
\begin{proof}
We prove the two inclusions. \\
~\newline
$(\subseteq)$ Let $x\in{\cal{SCR}}(\phi)$. By Theorem \ref{MAIN THEO 1}, this means that
$$x\in\bar{\Omega}(x) = \bigcap \{ B: \ B\text{ is strongly stable and } \omega(x)\subseteq B \}.$$
\noindent If $B$ is strongly stable and $\omega(x) \nsubseteq B$ then $x \in B^{\bullet}$. Therefore $x \in B \cup B^{\bullet}$ and the inclusion immediately follows. \\
~\newline
$(\supseteq)$ Let $x\in\bigcap\{ B \cup B^{\bullet}: \ B \text{ is strongly stable} \}$. Since $B \cap B^{\bullet} = \emptyset$, the next two cases occur.
If $x\in B$, then $\omega(x) \subseteq B$. On the other hand, if $x\in B^{\bullet}$ then, by definition, $\omega(x) \cap B = \emptyset$. This means that $x \in \bigcap \{ B: \ \omega(x) \subseteq B \} = \bar{\Omega}(x)$. Equivalently, $x \in {\cal{SCR}}(\phi)$.
\end{proof}
\noindent Clearly, decomposition (\ref{nostra}) equals to decomposition (\ref{altra}) when a dynamical system admits only attractors or, equivalently, $\cal{SCR}(\phi) = \cal{CR}(\phi)$. 
\begin{remark}
We remind that if $A_1 \ne A_2$ are attractors, then $A_1\cup A_1^*\neq A_2\cup A_2^*$. The same property does not hold true for strongly stable sets. Indeed, there could exist strongly stable sets $B_1 \ne B_2$ such that $B_1\cup B_1^{\bullet} = B_2\cup B_2^{\bullet}$. The obvious example is the trivial flow on a connected compact metric space: in such a case, any closed subset $B$ is strongly stable with $B^{\bullet} = X \setminus B$. As a consequence, Theorem \ref{teorema 5.3.1} can be rephrased as follows. Let us indicate by $\mathcal{SS}(X,\phi)$ the set of all strongly stable sets of $\phi: \R \times X \to X$. We define an equivalence relation $\sim$ on $\mathcal{SS}(X,\phi)$ by
$$B_1\sim B_2\quad\Longleftrightarrow\quad B_1\cup B_1^{\bullet}=B_2\cup B_2^{\bullet}$$
and we denote by $\mathcal{SS}(X,\phi) /\sim$ the set of the associated equivalence classes. By using this relation, decomposition (\ref{nostra}) equals to
\begin{equation} \label{nostra 1}
{\cal{SCR}}(\phi) = \bigcap \{ B \cup B^{\bullet} :\ B \in\mathcal{SS}(X,\phi)/\sim \}.
\end{equation}
\end{remark} 
For a $1$-dimensional dynamical system, we finally present an example of the above decomposition (\ref{nostra 1}). 
\begin{esempio}
\begin{figure}[h]
\centering
\begin{tikzpicture}
\draw (4,4) circle (20mm);
\draw [ultra thick] (4,6) arc (90:135:2);
\draw plot [mark=*] coordinates
	{(4,6)}; 
\node at (4,6.3) {B};
\draw plot [mark=*] coordinates
	{(2.586,5.414)};
\node at (2.286,5.5) {A};
\draw plot [mark=*] coordinates
	{(2.586,2.586)};
\node at (2.286,2.5) {D};
\draw plot [mark=*] coordinates
	{(5.414,2.586)};
\node at (5.714,2.5) {C};
\draw [-latex] (4,6) arc (90:45:2);
\draw [-latex] (6,4) arc (0:-90:2);
\draw [-latex] (2.586,2.586) arc (225:150:2);
\end{tikzpicture}
\caption{}
\label{figura_esempio_dec_non_repeller}
\end{figure}
\noindent On the circle $\mathbb{S}^1=\R / \Z$ equipped with the standard quotient metric, let us consider the dynamical system of Figure \ref{figura_esempio_dec_non_repeller}, where the bold line and the arrows denote respectively fixed points and the direction of the flow $\phi$. In such a case, ${\cal{CR}}(\phi)=\mathbb{S}^1$ and ${\cal{SCR}} (\phi)=Fix(\phi)$. Such a dynamical system does not present any attractor but $\mathcal{SS}(X,\phi)$ has four equivalence classes, corresponding to these cases:
$$B_0 \cup B_0^{\bullet} =cl(\widehat{AD}), \qquad B_1 \cup B_1^{\bullet} =cl(\widehat{DC}), \qquad 
B_2 \cup B_2^{\bullet} =cl(\widehat{CB}), \qquad B_3 \cup B_3^{\bullet} =\mathbb{S}^1.$$
\noindent Finally, taking the (finite) intersection, we obtain
\[
{\cal{SCR}}(\phi)=\bigcap \{B_i\cup B_i^{\bullet}:\ i=0,1,2,3\}=Fix(\phi).
\]
\end{esempio}

\addcontentsline{toc}{chapter}{Bibliography}
\bibliographystyle{plain}
\bibliography{Bibliography}

\end{document}